\documentclass[12pt,amstags,fleqn]{article}

\usepackage{amssymb}

\usepackage{amsmath}
\usepackage{amsthm}

\newtheorem{theorem}{THEOREM}[section]
\newtheorem{lemma}[theorem]{LEMMA}
\newtheorem{corollary}[theorem]{COROLLARY}
\newtheorem{example}[theorem]{EXAMPLE}
\newtheorem{definition}[theorem]{DEFINITION}

\newtheorem{remark}[theorem]{REMARK}

\newcommand{\mov}{\textnormal{Mov}} 
\newcommand{\fix}{\textnormal{Fix}} 

\usepackage{ifpdf}

\ifpdf
        \usepackage[pdftex]{graphicx}
\else
	\usepackage{epsfig}
        \usepackage{graphicx}
\fi

\title{Latin bitrades derived from groups}

\author{Nicholas J. Cavenagh\\
School of Mathematics \\
The University of New South Wales \\
NSW 2052 Australia
\\
\\
Ale\v s Dr\'apal \footnote{This work was supported by Australian Research Council Linkage International Award LX0453416 and institutional grant MSM0021620839}
\\
Department of Mathematics \\
Charles University \\
Sokolovsk\'a 83, 186 75 Praha 8 \\
Czech Republic
\\
\\
Carlo H\"{a}m\"{a}l\"{a}inen \\
Department of Mathematics \\
The University of Queensland \\
QLD 4072 Australia
}

\date{}

\begin{document}

\maketitle

\begin{abstract} 
A latin bitrade is a pair of
partial latin squares which are disjoint, occupy the same set of non-empty cells, and whose corresponding rows and columns contain the same set of entries. 
In (\cite{Dr9}) it is shown that a latin bitrade 
may be thought of as three 
derangements of the same set, 
whose product is the identity and whose cycles pairwise 
have at most one point in common. 
By letting a group act on itself by right translation,
we show how some latin bitrades may be derived directly from groups. 
Properties of latin bitrades such as
homogeneity, minimality (via thinness) and orthogonality may also
be encoded succinctly within the group structure.  
We apply the construction to some well-known groups, 
constructing previously unknown latin bitrades. 
In particular, we show the existence of minimal, $k$-homogeneous latin bitrades for each odd $k\geq 3$. In some cases these are the smallest known such examples.  
\end{abstract}

\section{Introduction}\label{s:intro}
One of the earliest studies of latin bitrades appeared in \cite{DrKe1},
where they are referred to as {\em exchangeable partial groupoids}.
Later (and at first independently), latin bitrades became of interest to
researchers of critical sets (minimal defining sets of latin squares)
(\cite{codose},\cite{Ke2},\cite{BaRe2}) and of the intersections
between latin squares (\cite{Fu}).  As discussed in \cite{wanless},
latin bitrades may be applied to the compact storage of large catalogues
of latin squares.  Results on other kinds of combinatorial trades
may be found in \cite{St} and \cite{KhMa2}.

In \cite{Dr9} it is shown that a latin bitrade may be thought of
as
a set of three permutations with no fixed points, whose product
is the identity and whose cycles have pairwise at most one point
in common. By letting a group act on itself by right translation, in
this paper we extend this  result to give a construction of
latin bitrades directly from groups. This construction does not give
every type of latin bitrade, however the latin bitrades generated in
this way are rich in symmetry and structure. Furthermore latin bitrade
properties such as orthogonality, minimality and homogeneity may
be encoded concisely into the group structure, as shown in Section 3.
Section 4 shows that many interesting examples can be constructed,
even from familiar examples of groups. Finally in Section 5 we give a table of known results of minimal $k$-homogeneous latin bitrades for small, odd values of $k$. 

Note that throughout this paper we compose permutations from left to
right. Correspondingly, if a permutation $\rho$ acts on a point $x$,
$x\rho$ denotes the image of $x$.  
Given a group $G$ acting on a set $X$, for each $g\in G$, 
$\fix(g)=\{x\mid x\in X,xg=g\}$ and 
$\mov(g)=\{x\mid x\in X,xg\neq g\}$. 
The group theory notation used
in this paper is consistent with most introductory texts, including
\cite{hall}.

\section{Permutation Structure}

\begin{definition}\label{defnLatinSquare}
Let $A_1$, $A_2$, and $A_3$ be finite, non-empty sets.
A {\em partial latin square\/} $T$
is an $|A_1| \times |A_2|$ array with
rows indexed by $A_1$, columns indexed by $A_2$,
and entries from $A_3$, such that 
each $e \in A_3$ appears at most once in each row
and at most once in
each column.
In this paper, we ignore unused rows, columns and symbols, so that 
$A_1$, $A_2$ and $A_3$ often have differing sizes. 
In the case where $|A_1|=|A_2|=|A_3|=n$ and 
each $e\in A_3$ appears exactly once in each row
and once in each column, we say that $T$ is a 
{\em latin square\/} of order $n$.
\end{definition}
We may view $T$ as a set and write $(x,\, y,\, z) \in T$
if and only if symbol $z$ appears in the cell at row $x$, column $y$.
As a binary operation we write
$x \circ y = z$ if and only if $(x,\, y,\, z) \in T(=T^{\circ})$.
Equivalently, a partial latin square $T^{\circ}$ is a subset
$T^{\circ} \subseteq A_1 \times A_2 \times A_3$ such that the following
conditions are satisfied:
\begin{itemize}
	\item[(P1)]
	If $(a_1,a_2,a_3),(b_1,b_2,b_3)\in T^{\circ}$, then either at most one of 
	$a_1=b_1$, $a_2=b_2$ and $a_3=b_3$ is true or all three are true.  
	\item[(P2)]
	The sets $A_1$, $A_2$, and $A_3$ are pairwise disjoint, and 
	for all $\alpha \in \bigcup_i A_i$, there exists an
	$(a_1, a_2, a_3) \in T^{\circ}$ with $\alpha=a_i$ for some $i$.
\end{itemize}
Let $T^{\circ},T^{\star}\subset A_1\times A_2\times A_3$ be two partial latin squares. 
Then $(T^{\circ},T^{\star})$ is called a {\em latin bitrade}
if the following 
conditions are all satisfied.
\begin{itemize}
\item[(R1)] $T^{\circ} \cap T^{\star} = \emptyset$.

\item[(R2)] For all $(a_1, a_2, a_3) \in T^{\circ}$ and all $r,s \in \{1, 2, 3\}$,
$r \neq s$, there exists a unique $(b_1, b_2, b_3) \in T^{\star}$
such that $a_r=b_r$ and $a_s=b_s$.

\item[(R3)] For all $(a_1, a_2, a_3) \in T^{\star}$ and all $r,s \in \{1, 2, 3\}$,
$r \neq s$, there exists a unique $(b_1, b_2, b_3) \in T^{\circ}$
such that $a_r=b_r$ and $a_s=b_s$.

\end{itemize}

Note that (R2) and (R3) imply that each row (column) of $T^{\circ}$
contains the same subset of $A_3$ as the corresponding row (column)
of $T^{\star}$.  We sometimes refer to $T^{\circ}$ as a {\em latin
trade} and $T^{\star}$ its {\em disjoint mate}.  The {\em size}
of a latin bitrade is equal to $|T^{\circ}|=|T^{\star}|$.

Given any two distinct latin squares $L^{\circ}$ and $L^{\star}$, 
each of order $n$, 
$(L^{\circ}\setminus L^{\star}, 
L^{\star}\setminus L^{\circ})$ is a latin bitrade.  
In this
way, latin bitrades describe the difference between two latin squares.
In fact, we may think of a latin trade as a subset of a latin square which may be replaced with a disjoint mate to obtain a new latin square. 

An {\em isotopism} of a partial latin square is a relabelling of
the elements of $A_1$, $A_2$ and $A_3$. 
Combinatorial properties of
partial latin squares are, in general, preserved under isotopism 
(in particular, any isotope of a latin bitrade is also a latin bitrade),
a fact we exploit in this paper.

\begin{example}
Let $A_1=\{a,b\}$, $A_2=\{c,d,e\}$ and $A_3=\{f,g,h\}$. 
Then $(T^{\circ},T^{\star})$ is a {\em latin bitrade}, where
$T^{\circ},T^{\star}\subset A_1\times A_2\times A_3$ are shown below:
\[
	T^{\circ} = ~\begin{tabular}{|c||c|c|c|}
	\hline $\circ$ & $c$ & $d$ & $e$ \\
	\hline \hline $a$ & $f$ & $g$ & $h$ \\
	\hline $b$ & $g$ & $h$ & $f$ \\
	\hline 
	\end{tabular}
	\qquad
	T^{\star} = ~\begin{tabular}{|c||c|c|c|}
	\hline $\star$ & $c$ & $d$ & $e$ \\
	\hline \hline $a$ & $g$ & $h$ & $f$ \\
	\hline $b$ & $f$ & $g$ & $h$ \\
	\hline 
	\end{tabular}.
\]
We may also write:
\begin{align*}
T^{\circ} & = \{(a,c,f),(a,d,g),(a,e,h),(b,c,g),(b,d,h),(b,e,f)\}\textnormal{ and } \\
T^{\star} & = \{(a,c,g),(a,d,h),(a,e,f),(b,c,f),(b,d,g),(b,e,h)\}. 
\end{align*} 
\label{egg1}
\end{example} 
It turns out (as shown in \cite{Dr9}) that latin bitrades may be defined 
in terms of permutations.
We first show how to derive permutations of $T^{\circ}$ from a given latin bitrade.

\begin{definition}
Define the map $\beta_r : T^{\star} \rightarrow
T^{\circ}$ where $(a_1,a_2,a_3)\beta_r = (b_1,b_2,b_3)$ implies that
$a_r \neq b_r$
and $a_i = b_i$ for $i \neq r$.
(Note that by conditions {\rm (R2)} and {\rm (R3)} the map $\beta_r$ and its
inverse are well defined.) 
In particular, let $\tau_1,\tau_2,\tau_3: T^{\circ} \rightarrow T^{\circ}$, where
$\tau_1=\beta_2^{-1}\beta_3$,
$\tau_2=\beta_3^{-1}\beta_1$ and
$\tau_3=\beta_1^{-1}\beta_2$.
For each $i\in \{1,2,3\}$, let ${\mathcal A}_i$ be the set of cycles in $\tau_i$. 
We will see these cycles as permutations of $T^{\circ}$.
\label{defya}
\end{definition}
 
\begin{example}
Consider the latin bitrade constructed in Example $\ref{egg1}$. Here:
\begin{align*}
\tau_1 &= ((a,c,f)(a,e,h)(a,d,g))((b,c,g)(b,d,h)(b,e,f)) \\
\tau_2 &= ((a,c,f)(b,c,g))((a,e,h)(b,e,f))((a,d,g)(b,d,h)) \textnormal{ and } \\
\tau_3 &= ((a,c,f)(b,e,f))((a,d,g)(b,c,g))((b,d,h)(a,e,h)). 
\end{align*}
\end{example}

\begin{lemma}
\label{perm}
The permutations $\tau_1$, $\tau_2$ and $\tau_3$ satisfy the following properties:
\begin{itemize}
\item[{\rm (Q1)}] If $\rho \in {\mathcal A}_r$, $\mu \in {\mathcal A}_s$, $1 \leq r < s \leq 3$,
then $\left| \mov(\rho) \cap \mov(\mu) \right| \leq 1$.

\item[{\rm (Q2)}] For each $i \in \{1, 2, 3\}$, $\tau_i$ has no fixed points.
\item[{\rm (Q3)}] $\tau_1\tau_2\tau_3=1$.
\end{itemize}
\end{lemma}

\begin{proof}
Observe that $\tau_i$ leaves the $i$--th coordinate of a triple fixed.  
\begin{itemize}
\item[{\rm (Q1)}] 
Let $r=1$, $s=2$ and take $\rho$, $\mu$ as specified.
Suppose that $x$ and $y$ are distinct points in 
$\mov(\rho)\cap \mov(\mu)$. Let $x = (x_1,x_2,x_3)$. Then
$(x_1,x_2,x_3)\rho^i = (x_1,x'_2,x'_3) = y$
and
$(x_1,x_2,x_3)\mu^j = (x''_1,x_2,x''_3) = y$
for some $i$, $j$. 
This implies that $x_2 = x'_2$, a contradiction
to the fact that $\rho$ leaves only the first co-ordinate fixed.
The cases $(r,s) = (1,3)$ and
$(2,3)$ are similar.

\item[{\rm (Q2)}] Each 
$\tau_i = \beta_s^{-1}\beta_r$ 
changes the $t$--th
component of a triple $x$, where $t \in \{ s,r\}$.

\item[{\rm (Q3)}] 
Observe that 
$\tau_1\tau_2\tau_3  
= \beta_2^{-1}\beta_{3}
\beta_{3}^{-1}\beta_{1}
\beta_{1}^{-1}\beta_{2}
= 1$.\qedhere
\end{itemize} 
\end{proof}

Thus from a given latin bitrade we may define a set of permutations with
particular properties. It turns out that there exists a 
reverse process.

\begin{definition}
\label{sdef}
Let $\tau_1$, $\tau_2$, $\tau_3$ be permutations on some set $X$ and for
$i\in \{1,2,3\}$, let ${\mathcal A}_i$ be the set of cycles of $\tau_i$. 
 Suppose that 
$\tau_1,\tau_2,\tau_3$  satisfy Conditions {\rm (Q1)}, {\rm (Q2)} and {\rm (Q3)} 
from Lemma $\ref{perm}$. Next, define 
\[
S^{\circ}= \{(\rho_1,\rho_2,\rho_3)\mid \rho_i\in {\mathcal A}_i
\textnormal{ and there exists $x$ such that $x \in \mov(\rho_i)$ for all $i$ } \}
\]
and
\[
\begin{array}{l} 
S^{\star}  =  
\{(\rho_1,\rho_2,\rho_3)\mid \rho_i\in {\mathcal A}_i, \textnormal{ there exist distinct
points $x$, $x'$, $x''$ in $X$ such that }  \\ 
x\rho_1=x', 
x'\rho_2=x'', 
x''\rho_3=x \}. 
\end{array}
\]
\end{definition}

\begin{theorem}[\cite{Dr9}]
\label{perm2}
Then the pair of partial latin squares
$(S^{\circ},S^{\star})$ is a latin bitrade of size $|X|$ with 
$|{\mathcal A}_1|$ rows, $|{\mathcal A}_2|$ columns and $|{\mathcal A}_3|$ entries. 
\end{theorem} 

\begin{proof}
Condition (Q1) ensures that $S^{\circ}$ is a partial latin square.

From (Q3), $\tau_1\tau_2\tau_3$ fixes every point in $X$. 
It follows that, for each point $x\in X$, there is a unique choice of 
$\rho_1\in {\mathcal A}_1$ not fixing $x$, 
$\rho_2\in {\mathcal A}_2$ not fixing $x\rho_1$, and
$\rho_3\in {\mathcal A}_3$ not fixing $x\rho_1\rho_2$,
such that $x\rho_1\rho_2\rho_3=x$.
Suppose that 
$(\rho_1,\rho_2,\rho_3), 
(\rho_1,\rho_2,\rho_3')\in S^{\star}$, where $\rho_3\neq \rho_3'$.  
Then from our previous observation, if $\rho_1\rho_2\rho_3$ fixes $x\in X$
and $\rho_1\rho_2\rho_3'$ fixes $x'\in X$, $x\neq x'$. 
But this implies that 
$x\rho_1,\, x'\rho_1 \in \mov(\rho_1)\cap \mov(\rho_2)$, 
contradicting (Q1). By symmetry,
$S^{\star}$ is also a partial latin square.   
 
Next, suppose that $(\rho_1,\rho_2,\rho_3)\in S^{\circ}\cap S^{\star}$.
Then there are distinct points $x,x',x''$ such that
$x\rho_1=x'$, 
$x'\rho_2=x''$ and 
$x''\rho_3=x$.
Thus $x' \in \mov(\rho_1) \cap \mov(\rho_2)$ and 
$x'' \in \mov(\rho_2) \cap \mov(\rho_3)$. 
But there exists $y \in \mov(\rho_1) \cap \mov(\rho_2) \cap \mov(\rho_3)$ 
and either $y\neq x'$ or $y\neq x''$ is true. Without loss of generality suppose that
$y\neq x''$ is true. Then $\left| \mov(\rho_2) \cap \mov(\rho_3) \right|
\geq 2$,
 contradicting (Q1).
Thus $S^{\circ}\cap S^{\star}=\emptyset$ and (R1) is satisfied.
 
Next we show that (R2) is satisfied. So suppose that 
$(\rho_1,\rho_2,\rho_3)\in S^{\circ}$ and let 
$y \in \mov(\rho_1) \cap \mov(\rho_2) \cap \mov(\rho_3)$. 
Then there is some $x$ and $z$ such that
$x\rho_1=y$ and $y\rho_2=z$. But $\rho_1$ is the only permutation in
${\mathcal A}_1$ that does not fix $x$ and $\rho_2$ is the only permutation in
${\mathcal A}_2$ that does not fix $y$. It follows from the observation in the
second paragraph of the proof that 
there is a unique $\rho_3'\in {\mathcal A}_3$ such that $z\rho_3' =x$. Thus 
$(\rho_1,\rho_2,\rho_3')\in S^{\star}$.
By symmetry (R2) is satisfied. 

Finally we show that (R3) is satisfied. 
So let $(\rho_1,\rho_2,\rho_3)\in S^{\star}$.
Then there are distinct points $x$, $x'$, $x''$ such that
$x\rho_1=x'$, 
$x'\rho_2=x''$ and 
$x''\rho_3=x$.
Then there exists $x' \in \mov(\rho_1) \cap \mov(\rho_2)$. 
Let $\rho_3'$ be the unique cycle of ${\mathcal A}_3$ that does not fix $x'$. 
Then $(\rho_1,\rho_2,\rho_3')\in S^{\circ}$.
By symmetry (R3) is satisfied. 

Since 
each cycle in $\rho_1$, $\rho_2$, $\rho_3$  gives rise to a unique row, column, entry (respectively),
the latin bitrade $(S^{\circ},S^{\star})$ will have
$|{\mathcal A}_1|$ rows, $|{\mathcal A}_2|$ columns and $|{\mathcal A}_3|$ entries. 
\end{proof}

\begin{example}
Let $\tau_1=(123)(456)$, $\tau_2=(14)(26)(35)$ and
$\tau_3=(16)(34)(25)$ be three permutations on the set
$\{1,2,3,4,5,6\}$. Then these permutations satisfy Conditions 
{\em (Q1)}, {\em (Q2)} and {\em (Q3)}, thus generating a latin bitrade of
size $6$ with two rows, three columns and three different entries.
In fact, this latin bitrade is isotopic to the latin bitrade
given in Example $\ref{egg1}$.
\end{example}

A latin bitrade is said to be {\em separated} if each row, column and 
entry gives rise to exactly one cycle of $\tau_1$, $\tau_2$ and $\tau_3$, 
respectively (see Definition \ref{defya}). 
The latin bitrade given in Example \ref{egg1} is separated. 
The next example gives a non-separated latin bitrade. 
\begin{example}
Let $A_1=\{a,b,c\}$, $A_2=\{d,e,f,g\}$ and $A_3=\{h,i,j,k\}$. 
Then $(T^{\circ},T^{\star})$ is a {\em latin bitrade}, where
$T^{\circ},T^{\star}\subset A_1\times A_2\times A_3$ are shown below:
\[
	T^{\circ} = 
~\begin{tabular}{|c||c|c|c|c|}
	\hline $\circ$ & $d$ & $e$ & $f$ & $g$ \\
	\hline \hline $a$ & $h$ & $i$ & $j$ & $k$ \\
	\hline $b$ & $i$ & $l$ & $k$ & \\
	\hline $c$ & $k$ & $j$ & $l$ & $h$ \\
	\hline 
	\end{tabular}
	\qquad
	T^{\star} = 
~\begin{tabular}{|c||c|c|c|c|}
	\hline $\circ$ & $d$ & $e$ & $f$ & $g$ \\
	\hline \hline $a$ & $i$ & $j$ & $k$ & $h$ \\
	\hline $b$ & $k$ & $i$ & $l$ & \\
	\hline $c$ & $h$ & $l$ & $j$ & $k$ \\
	\hline 
	\end{tabular}.
\]
Moreover, $(T^{\circ},T^{\star})$ is non-separated,  
by observation of row $c$.
\end{example}

The next theorem demonstrates that the process in 
Definition \ref{sdef} is the inverse of the process in Definition \ref{defya} for 
separated latin bitrades. 
\begin{theorem}
Let $(T^{\circ},T^{\star})$ be a separated latin bitrade. 
Let $\tau_1$, $\tau_2$ and $\tau_3$ be the corresponding set of permutations 
as given in Definition $\ref{defya}$. 
In turn, let $(S^{\circ},S^{\star})$ be the latin bitrade defined from 
$\tau_1$, $\tau_2$ and $\tau_3$ via Definition $\ref{sdef}$.
Then 
$(T^{\circ},T^{\star})$ and 
$(S^{\circ},S^{\star})$ are isotopic latin bitrades.
\end{theorem} 

\begin{proof}
For each 
$i\in \{1,2,3\}$ and each 
$x\in A_i$, 
let $f_i(x)\in {\mathcal A}_i$ be the 
cycle in $\tau_i$ which includes $x$ in each ordered triple. 
Since 
$(T^{\circ},T^{\star})$ is separated, each $f_i$ is a $1-1$ correspondence 
between 
$A_i$ and ${\mathcal A}_i$. 
Indeed, if $(x_1,x_2,x_3)\in T^{\circ}$, 
then $(x_1,x_2,x_3)\in 
\mov(f_1(x_1))\cap 
\mov(f_2(x_2))\cap 
\mov(f_3(x_3))$. 
Thus  
$(f_1(x_1),f_2(x_2),f_3(x_3))\in S^{\circ}$.

Next, let $(y_1,y_2,y_3)\in T^{\star}$. 
Let $(y_1,y_2,y_3)\beta_1 =(y_1',y_2,y_3)\in T^{\circ}$, 
where $y_1'\neq y_1$. Define
$y_2'$ and $y_3'$ similarly.
Then $\tau_1(=\beta_2^{-1}\beta_3)$ (in fact, its cycle $f_1(y_1)$) 
maps $(y_1,y_2',y_3)$ to 
$(y_1,y_2,y_3')$. 
Similarly 
$f_2(y_2)$ maps 
$(y_1,y_2,y_3')$ to $(y_1',y_2,y_3)$ and  
$f_3(y_3)$ maps 
$(y_1',y_2,y_3)$ to $(y_1,y_2',y_3)$. 
It follows that
$(f_1(y_1)$, $f_2(y_2)$, $f_3(y_3)) \in S^{\star}$.
\end{proof}

\begin{definition}
A latin bitrade $(T^{\circ}, T^{\star})$ is said to be {\em primary}  
if whenever $(U^{\circ},U^{\star})$ is a latin bitrade such that
$U^{\circ}\subseteq
T^{\circ}$ 
and
$U^{\star}\subseteq   
T^{\star}$, then 
$(T^{\circ}, T^{\star})=  
(U^{\circ}, U^{\star})$.  
\end{definition}  

It is not hard to show that a non-primary latin bitrade may be partitioned into 
smaller, disjoint latin bitrades.

\begin{definition}
A latin trade $T^{\circ}$ is said to be {\em minimal}  
if whenever $(U^{\circ},U^{\star})$ is a latin bitrade such that
$U^{\circ} \subseteq T^{\circ}$ 
then   
$T^{\circ}=  
U^{\circ}$.  
\end{definition}  

Note that for any primary bitrade $(T^{\circ},\, T^{\star})$, it is not
necessarily true that $T^{\circ}$ or $T^{\star}$ is a minimal trade.
Minimal latin trades are important in the study of {\em critical sets}
(minimal defining sets) of latin squares (see \cite{Ke2} for a recent survey). 

So a 
separated 
latin bitrade may be identified with a set of permutations that act
on a particular set $X$. Clearly, the permutations $\tau_1$, $\tau_2$,
$\tau_3$ generate some group $G$ which acts on the set $X$. We now study
the case where the set $X$ is the set of elements of $G$.

\begin{definition}
\label{tdef}
Let $G$ be a finite group. 
Let $a$, 
$b$, $c$ be non-identity elements of $G$ and let
$A=\langle a \rangle$,
$B=\langle b \rangle$ and
$C=\langle c \rangle$ 
such that:
\begin{itemize}
\item[{\rm (G1)}] 
 $abc=1$ and 
\item[{\rm (G2)}] 
$|A\cap B| = |A \cap C| = |B\cap C|=1$.
\end{itemize}
Next, define: 
\[
T^{\circ} = \{ (gA,gB,gC)\mid g\in G\}, \quad
T^{\star} = \{ (gA,gB,ga^{-1}C)\mid g\in G\}. 
\]
\end{definition}

\begin{theorem}
\label{main}
The pair of partial latin squares
$(T^{\circ},T^{\star})$ as defined above
is a latin bitrade with size $|G|$,
$|G:A|$ rows (each with $|A|$ entries), 
$|G:B|$ columns (each with $|B|$ entries)  and
$|G:C|$ entries (each occurring $|C|$ times).

If, in turn,
\begin{itemize}
\item[{\rm (G3)}] 
$\langle a,b,c \rangle = G$,
\end{itemize}
then the latin bitrade is primary.
\end{theorem} 

\begin{proof}
For each $g\in G$, define a map $r_g$ on the 
elements of $G$ by $r_g: x \mapsto xg$. 
Let $\tau_1=r_a$, $\tau_2=r_b$ and $\tau_3=r_c$.   
Then (G1) implies (Q3) and (G2) implies (Q1).  Since $a$, $b$ and $c$
are non-identity elements, each of $r_a, r_b$ and $r_c$ has no fixed
points, so (Q2) is also satisfied.  
So from Theorem~\ref{perm2} with $X=G$, there exists
a latin bitrade $(S^{\circ},S^{\star})$ defined in terms of the cycles
of $r_a$, $r_b$ and $r_c$.
A cycle in $r_a$ is of the form $(g, ga, ga^2, \ldots, ga^{\left| A
\right|-1})$
for some $g \in G$. Hence cycles of
$r_a$ (or $r_b$ or $r_c$) permute the elements of the left cosets of $A$ (or $B$ or $C$, respectively).   

Next relabel the triples of $S^{\circ}$ and $S^{\star}$,
replacing each cycle with its corresponding (unique) left coset. 
Let this (isotopic) latin bitrade be $(T^{\circ},T^{\star})$. 
Thus, from Definition \ref{sdef}, 
\begin{align*}
T^{\circ} &= \{(g_1A,g_2B,g_3C)\mid \left|g_1A\cap g_2B\cap g_3C\right|=1\} \textnormal{\  and }\\
T^{\star} &= \{(g_1A,g_2B,g_3C)\mid 
\textnormal{ there exist elements $h\in g_1A$, $h'\in g_2B$, $h''\in g_3C$} \\ 
& \qquad \textnormal{ such that $ha=h'$, $h'b=h''$, $h''c=h$} \}.
\end{align*}
Consider an element $(g_1A,g_2B,g_3C)\in T^{\circ}$.
Then there exists unique $g\in g_1A\cap g_2B \cap g_3C$. Thus 
$(g_1A,g_2B,g_3C)=(gA,gB,gC)$.
Next, consider $(g_1A,g_2B,g_3C)\in T^{\star}$. 
In terms of $h'$ we have
$g_1A = h'A$ and $g_3C = h' a^{-1}C$. 
Letting $g = h'$ we have
$(g_1A,g_2B,g_3C)=(gA,gB,ga^{-1}C)$.
Thus
$(T^{\circ},T^{\star})$ as given in Definition \ref{tdef} 
is a latin bitrade with size $|G|$.

From Theorem \ref{perm2}, the latin bitrade $(T^{\circ},T^{\star})$ 
has $|G:A|$ rows, $|G:B|$ columns and $|G:C|$ entries. 
We next show that each row has $|A|$ entries.  
Consider an arbitrary row $gA$
 in $(T^{\circ},T^{\star})$.  
Since $ga^iA \cap ga^iB \cap ga^iC = \{ ga^i \}$ for all $i$, we have
$(ga^iA, ga^iB, ga^iC) \in T^{\circ}$ for all $i$. These entries are
 actually all on the same row since $ga^iA = gA$ for all $i$. 
Next, suppose that
$ga^iB = ga^jB$ for some $i$, $j$.
Then $a^iB = a^jB$ so $a^{i-j} \in B$ and
$i = j$.  So there are at least $\left| A \right|$ elements in the row.
If there were more than $\left| A \right|$ elements in the row then there
 must be some $h$, $x$ such that $gA \cap hB = \{ x \}$. But then
$x = ga^i = hb^j$ for some $i$, $j$ and therefore
$h = ga^i b^{-j}$ so $hB = ga^i b^{-j} B = ga^i B$ and columns of this form have
already been accounted for.  
Similarly, each column has $|B|$ entries and each entry occurs $|C|$ times. 

Finally we have the (G3) condition. 
For the sake of contradiction, suppose that 
$(T^{\circ},T^{\star})$ is not primary
and
$G = \langle a, b, c \rangle$. 
Then 
there is a (non-empty) bitrade
$(W^{\circ},W^{\star})$
 such that 
 $W^{\circ} \subset T^{\circ}$
 and 
 $W^{\star} \subset T^{\star}$. 
 Suppose that column $gB$ lies in $(W^{\circ},W^{\star})$ for some 
$g\in G$.  

 To avoid a notation clash with
 the $\tau_i$, define 
 $\nu_1=\beta_2^{-1}\beta_3$,
 $\nu_2=\beta_3^{-1}\beta_1$ and
 $\nu_3=\beta_1^{-1}\beta_2$ where the $\beta_r$ send
 $T^{\star}$ to $T^{\circ}$.
 Since 
$(W^{\circ},W^{\star})$
is a subtrade, the permutations $\nu_i$ send elements of 
$W^{\circ}$ to itself.
 The
 cycles of $\nu_i$ are of 
 length $\left| A \right|$, 
 $\left| B \right|$, or
 $\left| C \right|$, so
 $W^{\circ}$ has 
 $\left| A \right|$ entries per row 
 and
 $\left| B \right|$ entries per column.  
 In particular, if column $gB$ intersects $(W^{\circ},W^{\star})$, then 
 column $ga^iB$ intersects $(W^{\circ},W^{\star})$ for any $i$. 
 By a similar analysis of the rows, if row $gA$ lies in
 $(W^{\circ},W^{\star})$ then 
 row $gb^jA$ lies in $(W^{\circ},W^{\star})$ for any $j$. 
 It follows that 
 \[\begin{array}{llll}
 & (gA,gB,gC)\in W^{\circ} & \Rightarrow & (ga^iA=gA,ga^iB,ga^iC)\in W^{\circ} \\
 \Rightarrow & (ga^ib^jA,ga^ib^jB=ga^iB,ga^ib^jC)\in W ^{\circ} & \Rightarrow 
 & (ga^ib^jA,ga^ib^ja^kB,ga^ib^ja^kC)\in W^{\circ},
 \end{array}\]
 for any $i$, $j$ and $k$. Thus if column $gB$ 
 lies in $W^{\circ}$, then any column of the form $ga^ib^ja^kB$ lies in $W^{\circ}$. 
 By an iterative process, since any element of $G$ can be written as a product of powers of $a$ and $b$, 
 it follows that $W^{\circ}$ includes every column of $T^{\circ}$ and
 $(W^{\circ},W^{\star}) = (T^{\circ},T^{\star})$. 
\end{proof}

\begin{corollary}
\label{handy}
The latin trade $T^{\circ}$ in the previous theorem is equivalent to the 
set $\{(g_1A,g_2B,g_3C)\mid g_1,g_2,g_3\in G, \left|g_1A\cap g_2B\cap
g_3C \right| = 1\}$, 
which is in turn equivalent to 
$\{(g_1A,g_2B,g_3C)\mid g_1,g_2,g_3\in G, g_1A\cap g_2B =\{g_3\}\}$. 
\end{corollary}

It should be noted that this construction does not produce {\em every}
latin bitrade, as latin bitrades in general may have rows and columns
with varying sizes.  However, as the rest of the paper demonstrates,
this technique produces 
many interesting examples.

\begin{example}
Let $G = \langle s, t \mid s^3 = t^2 = 1, ts = s^2t \rangle$, the 
symmetric group on three letters. 
$A = \langle s \rangle$,
$B = \langle t \rangle$,
$C = \langle ts^2 \rangle$. Then {\rm (G1)}, {\rm (G2)} and {\rm (G3)} are satisfied.
The left cosets are:
$\langle s \rangle $, $t \langle s \rangle $;
$\langle t \rangle $, $s \langle t \rangle $, $s^2 \langle t \rangle $;
$\langle ts^2 \rangle $, $s \langle ts^2 \rangle $, $s^2 \langle ts^2 \rangle $.
Then the latin bitrade is 
\[
T^{\circ}  =
\begin{tabular}{|c||c|c|c|}
\hline $\circ$ & $B$ & $sB$ & $s^2B$ \\
\hline \hline $A$ & $C$ & $sC$ & $s^2C$\\
\hline $tA$ & $s^2C$ & $C$ & $sC$\\
\hline
\end{tabular}
\quad
T^{\star} = 
\begin{tabular}{|c||c|c|c|}
\hline $\star$ & $B$ & $sB$ & $s^2B$ \\
\hline \hline $A$ & $s^2C$ & $C$ & $sC$\\
\hline $tA$ & $C$ & $sC$ & $s^2C$\\
\hline
\end{tabular}.
\]
Note that this latin bitrade is isotopic to the one given in Example $\ref{egg1}$. 
\end{example}

\section{Orthogonality, minimality and homogeneity}

In this section we describe how certain properties of latin bitrades constructed as in Theorem~\ref{main} may be 
encoded in the group structure. 

\begin{definition}
A latin bitrade $(T^{\circ}, T^{\star})$ is said to be {\em orthogonal}  
if whenever $i\circ j = i'\circ j'$ (for $i\neq i'$, $j\neq j'$), then
$i\star j\neq i' \star j'$.
\end{definition}  

We use the term orthogonal because if 
$(T^{\circ}, T^{\star})$ is a latin bitrade and 
$T^{\circ}\subset L_1$, 
$T^{\star}\subset L_2$, where $L_1$ and $L_2$ are {\em mutually orthogonal
latin squares} (see \cite{crc} for a definition), then 
$(T^{\circ}, T^{\star})$ is {\em orthogonal}.

\begin{lemma} \label{lem:orthog}
A latin bitrade $(T^{\circ}, T^{\star})$ constructed from a
group $G = \langle a, b, c \rangle$ as in Theorem~$\ref{main}$ 
is orthogonal if and only if $|C \cap C^{a}| = 1$.
\end{lemma}

\begin{proof}
First suppose that the latin bitrade is not orthogonal.
Then $gC = hC$ and $ga^{-1}C = ha^{-1}C$ for some $g,h\in G$ with $g\neq h$,
as shown in the following diagram: 
\begin{center}
\begin{tabular}{ccc}

	\begin{tabular}{c|cc}
	$\circ$ & $gB$ & $hB$ \\
	\hline $gA$    & $gC$ &      \\
	       $hA$    &      & $hC$
	\end{tabular}
&
$\quad$
&
	\begin{tabular}{c|cc}
	$\star$ & $gB$ & $hB$ \\
	\hline $gA$    & $ga^{-1}C$ &      \\
	       $hA$    &      & $ha^{-1}C$
	\end{tabular}
\end{tabular}
\end{center}
Then
$g^{-1}h \in C$ and $ag^{-1}ha^{-1} \in C$ which 
implies that $g^{-1}h \in a^{-1}C a = C^{a}$. Thus
$|C\cap C^{a}|\neq 1$.

Conversely, suppose that $x\in C\cap C^{a}$ where 
$x\neq 1$. Then 
we may write $x=h^{-1}g$ for some non-identity elements $g,h\in G$.
We then reverse the steps in the previous paragraph to show that the latin bitrade is not orthogonal.
\end{proof}

Is it possible to encode minimality via our group construction?
We do this by encoding a ``thin'' property of latin bitrades, which,
together with the primary property, implies minimality.

\begin{definition}
\label{defn:thin}
A latin bitrade $(T^{\circ}, T^{\star})$ is said to be {\em thin}  
if whenever $i\circ j = i'\circ j'$ (for $i\neq i'$, $j\neq j'$), then
$i\star j'$ is either undefined, or 
$i\star j'= i\circ j$.
\end{definition}

\begin{lemma}
Let 
$(T^{\circ}, T^{\star})$ be a thin and primary latin bitrade.  
Then $T^{\circ}$  is a minimal latin trade.  
\label{tpm}
\end{lemma}

\begin{proof}
Suppose, for the sake of contradiction, that
$(T^{\circ}, T^{\star})$
is thin but not minimal.  
Then there exists a latin bitrade 
$(U^{\circ}, U^{\otimes})$ such that
$U^{\circ}\subset T^{\circ}$.
Since $(U^{\circ}, U^{\otimes})$ is a latin bitrade, then 
for any $i$, $j$, $k$ such that
$i \otimes j = k$,
there are $i'$, $j'$ such that
$i \circ j' = i' \circ j = k$,
where $i \neq i'$ and $j \neq j'$.
By thinness of 
$(T^{\circ}, T^{\star})$, it follows that
$i \star j$ is either undefined or $i \star j = k$. However,
$i \otimes j$ is defined, so $i \circ j$ is defined in both $U^{\circ}$
and $T^{\circ}$. So $i \star j = k$ and we see that
$U^{\otimes} \subset T^{\star}$, contradicting the primary property.
\end{proof}

In general, the minimality of latin bitrades can be complicated 
to check, (see, for example, \cite{CaDoYa})  highlighting the elegance
of the following lemma.

\begin{lemma} \label{lem:thingroupbitrade}
A latin bitrade $(T^{\circ}, T^{\star})$ constructed from a
group $G = \langle a, b, c \rangle$ as in Theorem~$\ref{main}$ 
is {\em thin} (and thus minimal) if and only if
the only solutions to the equation
$a^i b^j c^k = 1$ are
$(i,j,k) = (0,0,0)$ and 
$(i,j,k) = (1,1,1)$, where $i$, $j$ and $k$ are calculated 
modulo $|A|$, $|B|$ and $|C|$, respectively.
\end{lemma}

\begin{proof}
We first rewrite Definition~\ref{defn:thin} in terms of cosets. 
Let $i = g_1A$, $j = g_2B$ for some $g_1$, $g_2 \in G$. Since $i \circ j$ must
be defined it follows that $g_1A$ and $g_2B$ intersect. By
Corollary~\ref{handy} this intersection is a unique element $g \in G$.
Thus $i = gA$, $j = gB$. Similarly, $i' = hA$, $j' = hB$ for a unique $h
\in G$. For a latin bitrade to be thin we must have
$i \star j' = i \circ j$ whenever $i \star j'$ is defined. In other
words, the latin bitrade is thin if $gC=hC$ implies that
$gC=xa^{-1}C$ whenever there exists (a unique) $x \in gA \cap hB$.

\begin{center}
\begin{tabular}{ccc}
	\begin{tabular}{c|cc}
	$\circ$ & $gB$ & $hB$ \\
	\hline $gA$    & $gC$ & $xC$     \\
	       $hA$    &      & $hC$
	\end{tabular}
&
$\quad$
&
	\begin{tabular}{c|cc}
	$\star$ & $gB$ & $hB$ \\
	\hline $gA$    & $ga^{-1}C$ &  $xa^{-1}C$    \\
	       $hA$    &      & $ha^{-1}C$
	\end{tabular}
\end{tabular}
\end{center} 

First suppose that the only solutions to $a^i b^j c^k = 1$ are $(i,j,k)
= (0,0,0)$ and $(i,j,k) = (1,1,1)$. To check for thinness, suppose that
$gC = hC$ and that there exists an $x \in gA \cap hB$. Then
$x = ga^{m} = hb^{-n}$ for some $m$, $n$.
Now
$a^{m} b^{n} = g^{-1} h = c^{-p}$
for some $p$ since $gC = hC$ implies that $g^{-1}h \in C$.
So $a^{m} b^{n} c^{p} = 1$. 
If $m=n=p=0$ then $g = h$ which is a contradiction. Otherwise
$(m,n,p) = (1,1,1)$ so $x = ga$. Now
$g^{-1} xa^{-1} = g^{-1} gaa^{-1} = 1 \in C$ so $gC = xa^{-1}C$
as required.

Conversely, suppose that the latin bitrade is thin and that
$a^m b^n c^p = 1$ for some $m$, $n$, $p$. There is always a trivial
solution $(0,0,0)$ so it suffices to check that $(1,1,1)$ is the only
other possibility.
Since $a^m b^n c^p = 1$ we can 
write
$g a^m = (g c^{-p}) b^{-n}$
for any $g \in G$.
Define $h = g c^{-p}$ and $x = g a^m = h b^{-n}$.
Now $hC = g c^{-p} C = gC$ and by definition
$x \in gA \cap hB$ so $gA \star hB$ is defined. Now thinness implies
that $gC = xa^{-1}C$, so $g^{-1}xa^{-1} \in C$. 
Then $g^{-1} g a^m  a^{-1} = a^{m-1} \in C$ so $m=1$.
Since $abc=1$,
\[
1 = a b^n c^p = c^{-1} b^{-1} b^n c^p = c^{-1} b^{n-1} c^p
\Rightarrow b^{n-1} c^{p-1} = 1
\]
so $b^{n-1} = c^{1-p}$ and therefore $(m,n,p) = (1,1,1)$. 
\end{proof}

\begin{definition}
A latin trade $T^{\circ}$ is said to be {\em ($k$-)homogeneous} if each
row and column contains precisely $k$ entries and each entry occurs
precisely $k$ times within $T^{\circ}$.  
\end{definition}

The next lemma follows from Theorem~\ref{main}.

\begin{lemma}
A latin bitrade $(T^{\circ}, T^{\star})$ constructed from a
group $G = \langle a, b, c \rangle$ as in Theorem~$\ref{main}$ 
is $k$-homogeneous if and only if $|A|=|B|=|C|=k$.
\end{lemma}

A $2$-homogeneous latin bitrade is trivially the union of latin squares of order $2$. 
A construction for $3$-homogeneous latin bitrades is given in \cite{CaDoDr}; 
moreover in \cite{Ca20} it is shown that this construction gives every possible primary $3$-homogeneous latin bitrade. 
The problem of determining the spectrum of sizes of $k$-homogeneous latin bitrades is solved in \cite{BBKM}; however if we add the condition 
of minimality this problem becomes far more complex.
Some progress towards this has been made in \cite{CaDoDr2} and \cite{CaDoYa}; however it is even an open problem to determine the possible sizes of a minimal $4$-homogeneous latin bitrade. 
The theorems in the following section yield previously unknown cases of minimal $k$-homogeneous latin bitrades.   

\section{Examples}

In this section we apply Theorem~\ref{main} to generate bitrades from
various groups.  All of the bitrades constructed will be primary, so
by Lemma~\ref{tpm} thinness will imply minimality for each example.

\subsection{Abelian groups}

An abelian group $G$ has the normaliser $N_G(C)$ equal to the entire
group so $C = C^{a}$. By Lemma~\ref{lem:orthog} abelian groups will not
generate orthogonal bitrades. The next lemma gives an example of a latin
bitrade constructed from an abelian group.

\begin{lemma}
Let $p$ be a prime.
Then $G = ({\mathbb Z}_p \times {\mathbb Z}_p,+)$ generates a latin bitrade
$(T^{\circ}, T^{\star})$ 
using $a=(0,1)$, $b=(1,0)$ and $c=(p-1,p-1)$.
\end{lemma}

\begin{proof}
First, $(0,1)+(1,0)+(p-1,p-1) = (0,0)$ so (G1) is met. For (G2):
\begin{itemize}

\item $A \cap B = \{ (0,0) \}$.

\item Note that $(0,x) \in C$ only when $x = 0$ so $A$ and
$C$ intersect in the single element $(0,0)$.

\item There is no $(x,0) \in B \cap C$ with $x \neq 0$ by similar
reasoning.
\end{itemize}
Lastly, $G = \langle a, b, c \rangle$ follows from the definition of $a$
and $b$ so (G3) is satisfied.  
\end{proof}

The latin bitrade $T^{\circ}$ in the above lemma is in fact a latin square so in some sense this example is degenerate.

\subsection{A $p^3$-group example}

It is well known (see, for example, \cite{hall}~p. 52) 
that for any odd prime $p$ there exists  
a non-abelian group $G$ of order $p^3$, with
generators $a$, $b$, $c$ and relations
\begin{align}
&a^p = b^p = c^p = 1, \label{eqn:nonabelian:1}\\ 
&ab = bac, \label{eqn:nonabelian:2}\\
&ca = ac, \label{eqn:nonabelian:3}\\
&cb = bc. \label{eqn:nonabelian:4}
\end{align}
For convenience we let $z = c^{-1}$ throughout this section.

\begin{lemma}\label{lem:canonicalp3}
Any word $w \in G$ can be written in the form
$a^i b^j c^k$. Further, the group operation can be defined in terms of
the canonical representation
\[
(a^i b^j z^k) (a^r b^s z^t) = a^{i+r} b^{j+s} z^{k+t+jr}.  
\]
\end{lemma}

\begin{lemma} \label{lem:gammak}
Let $\gamma = b^{-1} a^{-1}$. Then
$\gamma^k = a^{-k} b^{-k} z^{k(k+1)/2}$ and $\gamma$ has 
order $p$.
\end{lemma}

\begin{proof} 
First we show that  
$(a^{-1} b^{-1})^k = a^{-k} b^{-k} z^{k(k-1)/2}$ 
by induction on $k$. When $k=1$ the statement is true. The inductive
step is:
\begin{align*}
(a^{-1} b^{-1})^{k+1} &= (a^{-1} b^{-1})  (a^{-1} b^{-1})^{k} 
= (a^{-1} b^{-1})  a^{-k} b^{-k} z^{k(k-1)/2} \\
&= a^{-(k+1)} b^{-(k+1)} z^{k(k-1)/2 + k} 
= a^{-(k+1)} b^{-(k+1)} z^{(k+1)k/2}.
\end{align*}
Now we can evaluate $\gamma^k$:
\begin{align}
\gamma^k 	&= (b^{-1} a^{-1})^k = (a^{-1} b^{-1} z)^k = (a^{-1} b^{-1})^k z^k  \nonumber \\
         	&= a^{-k} b^{-k} z^{k + k(k-1)/2} 
		= a^{-k} b^{-k} z^{k(k+1)/2}. \nonumber
\end{align} 
Since
$\gamma^p = a^{-p} b^{-p} z^{p(p+1)/2} = 1$ we see that
$\gamma$ has order $p$.  
\end{proof}

\begin{theorem} \label{p3}
Let $\alpha = a$, $\beta = b$, $\gamma = b^{-1} a^{-1}$
where $a$, $b$, and $c$ generate a group satisfying
\eqref{eqn:nonabelian:1} through \eqref{eqn:nonabelian:4}.  
Then $\alpha$, $\beta$ and $\gamma$ satisfy conditions {\rm (G1)}, {\rm (G2)} and {\rm (G3)} of Theorem~$\ref{main}$. 
Thus for each prime $p$, there exists a primary, $p$-homogeneous latin bitrade of size $p^3$
given by 
\[(
\{
(g\langle \alpha \rangle,
g\langle \beta \rangle,
g\langle \gamma \rangle) \mid g\in G \}, 
\{
(g\langle \alpha \rangle,
g\langle \beta \rangle,
g\alpha^{-1}\langle \gamma \rangle) \mid g\in G \}).\] 
\end{theorem}

\begin{proof}
By definition $\alpha \beta \gamma = 1$ so {\rm (G1)} is true. For (G2):
\begin{itemize}

\item The element $a$ is of order $p$ so any non-identity element of
$\langle a \rangle$ generates $\langle a \rangle$. The same holds for
$b$ and $\langle b \rangle$. If 
$a^m = b^n$ for some $0 < m, n < p$ then
$\langle a \rangle = \langle b \rangle$.
So $b = a^r$ for some $r$ and~\eqref{eqn:nonabelian:2} 
becomes
$a^{r+1} = a^{r+1}c$ so 
$c=1$, a contradiction. Hence 
$\langle \alpha \rangle \cap \langle \beta \rangle = 1$.
 
\item  The subgroup $\langle a \rangle$ has order $p$ and by
Lemma~\ref{lem:gammak} so does $\langle \gamma \rangle$.  If $a^l =
\gamma^k$ for some $0 < l,k < p$ then 
$\langle a \rangle = \langle \gamma \rangle$. The argument is now
similar to the first case.
Hence $\langle \alpha \rangle \cap \langle \gamma \rangle = 1$.

\item Showing that
$\langle \beta \rangle \cap \langle \gamma \rangle = 1$ is very similar
to the first case.

\end{itemize}
Lastly,
$G = \langle \alpha, \beta, \gamma \rangle$ since
$a = \alpha$, $b = \beta$, and 
$c = \alpha^{-1}\beta^{-1}\gamma^{-1}$.
Thus (G3) is satisfied.
\end{proof}

\begin{lemma} \label{lemma:ab:centre}
Let $G$ be a finite group and $g$, $h \in G$. If the product $gh$ is in  $Z(G)$ then 
$gh = hg$.
\end{lemma}

\begin{proof}
Since $gh \in Z(G)$ it must commute with any element of $G$. Thus
$(gh) g^{-1} = g^{-1} (gh) = h$
so $gh = hg$.
\end{proof}

\begin{lemma} \label{lemma:p3:centre}
Let $G$ be the group defined by~\eqref{eqn:nonabelian:1}
through~\eqref{eqn:nonabelian:4}. Then
$Z(G) = \langle c \rangle$.
\end{lemma}

\begin{proof}
Since $c$ is in $Z(G)$ we know that $\langle c \rangle \leq Z(G)$. For
the converse,
suppose that there is a
group element $w$ in $Z(G)$ but $w \notin \langle c \rangle$.
By Lemma~\ref{lem:canonicalp3} we can write 
$w = a^i b^j c^k$ for some $i$, $j$, $k$.
Then $(a^i b^j c^k) c^{-k} \in Z(G)$ 
since $c \in Z(G)$, which means that
$a^i b^j \in Z(G)$.
By Lemma~\ref{lemma:ab:centre}, $a^i b^j = b^j a^{i}$.  
However, using~\eqref{eqn:nonabelian:2} we
have $a^i b^j = b^j a^i c^{ij}$ so it must be that $p \mid ij$. 
If $p \mid i$ then $w = b^j c^k$ which implies that 
$b \in Z(G)$, a contradiction. Similarly, if $p \mid j$ then
$a \in Z(G)$, another contradiction.  
\end{proof}

\begin{lemma}
Latin bitrades constructed as in Theorem~$\ref{p3}$ are thin.
\end{lemma}

\begin{proof}
Suppose that $\alpha^i \beta^j \gamma^k = 1$ for some $i$, $j$, $k$.
Then
by Lemma~\ref{lem:canonicalp3} and \ref{lem:gammak}
\begin{align}
1 &= a^i b^j \gamma^k = a^i b^j \left( a^{-k} b^{-k} z^{k(k+1)/2} \right) \nonumber \\
  &= a^{i-k} b^{j-k} z^{{k(k+1)/2} - jk}. \label{eqn:thin:p3}
\end{align}
By Lemma~\ref{lemma:p3:centre},
$a^{i-k} b^{j-k} \in Z(G)$
and by Lemma~\ref{lemma:ab:centre},
$a^{i-k} b^{j-k} = b^{j-k} a^{i-k}$.
Using~\eqref{eqn:nonabelian:2} a total of $(i-k)(j-k)$ times we have
$a^{i-k} b^{j-k} = b^{j-k} a^{i-k} c^{(i-k)(j-k)}$ so
$c^{(i-k)(j-k)} = 1$. Since $p$ is prime there are
two cases:
\begin{enumerate}
\item If $p \mid i-k$ then~\eqref{eqn:thin:p3} reduces to
$1 = b^{j-k} z^{ {k(k+1)/2}- jk}$ so
$p \mid j-k$.

\item If $p \mid j-k$ then~\eqref{eqn:thin:p3} reduces to
$1 = a^{i-k} z^{ {k(k+1)/2} - jk}$ so
$p \mid i-k$.  
\end{enumerate} 
Thus $p \mid i-k$ and $p \mid j-k$. Now
$i \equiv k \pmod{p}$ and $j \equiv k \pmod{p}$ which implies
$i\equiv j\equiv k \pmod{p}$, and~\eqref{eqn:thin:p3}
becomes
\[
1 = z^{ {k(k+1)/2} - k^2} \quad \Rightarrow \quad 1 = z^{ k(k-1)/2}, 
\]
so $p \mid \frac{1}{2}k(k-1)$.
If $p \mid k$ then $(i,j,k) \equiv (0,0,0)$. Otherwise,
$p \mid \frac{1}{2}(k-1)$ and 
$(i,j,k) \equiv (1,1,1)$. By Lemma~\ref{lem:thingroupbitrade} the latin bitrade is thin.
\end{proof}

\begin{lemma}
The latin bitrade constructed in Theorem~$\ref{p3}$ is orthogonal.
\end{lemma}

\begin{proof}
Suppose for the sake of contradiction (see Lemma \ref{lem:orthog}) that 
$\langle \gamma \rangle$ and $\langle \gamma \rangle^{\alpha}$ have a
nontrivial intersection.
Since $\langle \gamma \rangle$ 
and
$\langle \gamma \rangle^{\alpha}$ 
are cyclic of order $p$, it must be that
$\langle \gamma \rangle = \langle \gamma \rangle^{\alpha}$ so
$\alpha^{-1} \gamma^k \alpha = \gamma$ for some $k$. Hence
\[\begin{array}{lrcl}
& \alpha^{-1} \gamma^k \alpha &= &  
b^{-1} a^{-1} \\
\Rightarrow \quad \quad \quad & (ab) a^{-1} \left( a^{-k} b^{-k} z^{{k(k+1)/2} } \right) a &= & 1 \\
\Rightarrow \quad & ab a^{-k} b^{-k} z^{{k(k+1)/2}  - k} &= & 1 \\
\Rightarrow \quad & a^{-k+1} b^{-k+1} z^{{k(k+1)/2}  - 2k} &= & 1.
\end{array}\]
Now $a^{-k+1} b^{-k+1} \in Z(G)$ so by
Lemma~\ref{lemma:ab:centre},
$a^{-k+1} b^{-k+1} = b^{-k+1} a^{-k+1}$. From~\eqref{eqn:nonabelian:2}
we can deduce that 
$a^{-k+1} b^{-k+1} = b^{-k+1} a^{-k+1} c^{(-k+1)(-k+1)}$ hence
$k \equiv 1 \pmod{p}$. Now
\[
a^{-1} \gamma a = \gamma \quad \Rightarrow \quad  
a^{-1} b^{-1} a^{-1} a = b^{-1} a^{-1} 
\quad 
\Rightarrow \quad a b = b a,\]
which is a contradiction since $a$ and $b$ do not commute with each
other.
\end{proof}

\subsection{ $|G|=pq$, where $p$ and $q$ are primes and $G$ is non-abelian.}

Let $p$ and $q$ be primes such that $p>q>2$ and $q$ divides $p-1$. 
Let $G=\langle a,b \rangle$ be the non-abelian group of order $pq$ 
defined by:
\begin{align}
&a^p = b^q = 1 \textnormal{ \ \ and } \label{eqn:pq:1}\\ 
&b^{-1}ab = a^r, \label{eqn:pq:2}
\end{align}
where $r \in \{2,3,\ldots, p-1\}$
is some solution to $r^q\equiv 1 \pmod{p}$.
The following remark may be verified by induction. 

\begin{remark}
Let $G$ be the group defined by~\eqref{eqn:pq:1} and \eqref{eqn:pq:2}.
Then for any integers $i$, $j$, $k$, $l$:
\begin{align}
&b^ia^jb^ka^l=b^{i+k}a^m, \label{eqn:pq:3} \\ 
&(ab)^k = b^k a^{r(r^k-1)/(r-1)}, \label{lem:pq:abt} \\
&(bab)^k = b^{2k} a^{ r (r^{2k} - 1) / (r^2 - 1)  }, \label{eqn:bab} 
\end{align}
where $m=jr^k+l$.
\end{remark}

\begin{theorem} \label{pq} 
Let $\alpha=b$, $\beta=ab$ and $\gamma=b^{-1}a^{-1}b^{-1}$
where $a$ and $b$ generate a group that satisfies
\eqref{eqn:pq:1} and \eqref{eqn:pq:2}.
Then $\alpha$, $\beta$ and $\gamma$ satisfy conditions {\rm (G1)}, {\rm (G2)} and {\rm (G3)} of Theorem~$\ref{main}$. 
Thus for each pair of primes $p,q$ such that $q>2$ and $q$ divides $p-1$, there exists a $q$-homogeneous latin bitrade of size $pq$
given by 
\[(
\{
(g\langle \alpha \rangle,
g\langle \beta \rangle,
g\langle \gamma \rangle) \mid g\in G \}, 
\{
(g\langle \alpha \rangle,
g\langle \beta \rangle,
g\alpha^{-1}\langle \gamma \rangle) \mid g\in G \}).\] 
\end{theorem}

\begin{proof}
Clearly
$\alpha\beta\gamma=1$, thus satisfying (G1). 
It is also clear that $\alpha$ and $\beta$ together generate $G$, so $(G3)$ is satisfied.
We next check $(G2)$. 
\begin{itemize}
\item If 
$\langle \alpha \rangle \cap
\langle \beta \rangle \neq 1$,  
then $b^k=ab$ for some $k$. This implies that
$b^{k-1} = a$  so $a \in \langle b \rangle$. 
Then $a$ and $b$ must generate the same cyclic
subgroup of prime order, a contradiction since $p \neq q$.

\item 
If 
$\langle \alpha \rangle \cap
\langle \gamma \rangle \neq 1$, then
$b^k=b^{-1}a^{-1}b^{-1}$ for some $k$, or equivalently
$a^{-1}=b^{k+2}$. So $a \in \langle b \rangle$, a contradiction.

\item If 
$\langle \beta \rangle \cap
\langle \gamma \rangle \neq 1$,  
$(ab)^k=b^{-1}a^{-1}b^{-1}$ for some $k$.
Then
$(ab)^{k+1}=b^{-1}$. Therefore from Equation~\ref{lem:pq:abt} we have
$b^{k+1} a^{ r(r^{k+1}-1)/(r-1) } = b^{-1}$
so
$b^{k+2} a^{r(r^{k+1}-1)/(r-1)} = 1$. The subgroups 
$\langle a \rangle$
and
$\langle b \rangle$
only intersect in the identity, so $k\equiv -2 \pmod{q}$.
But
\[
(ab)^q =b^qa^{r(r^q-1)/(r-1)} = b^q = 1
\]
as $r^q\equiv 1 \pmod{p}$. 
Thus
$(ab)^k=(ab)^{-2} = b^{-1} a^{-1} b^{-1}$ which implies that $a=1$, a contradiction.
\end{itemize} 
It remains to show that this latin bitrade is $q$-homogeneous. 
To see this, note that
$\beta^q = (ab)^q = 1$.
Next, from Equation~\ref{eqn:bab},
\[\gamma^q = (b^{-1}a^{-1}b^{-1})^q = ((bab)^q)^{-1}=
\left( b^{2q}a^{ r (r^{2q} - 1) / (r^2 - 1)} \right)^{-1}= 1.\]
\end{proof}

\begin{lemma}
The latin bitrade constructed in Theorem~$\ref{pq}$ is orthogonal.
\end{lemma}

\begin{proof}
Suppose for the sake of contradiction (see Lemma \ref{lem:orthog}) that 
$\langle \gamma \rangle$ and $\langle \gamma \rangle^{\alpha}$ have a
nontrivial intersection.
Since $\langle \gamma \rangle$ 
and
$\langle \gamma \rangle^{\alpha}$ 
are cyclic of order $q$, it must be that
$\langle \gamma \rangle = \langle \gamma \rangle^{\alpha}$ so
$\alpha^{-1} \gamma^k \alpha = \gamma$ for some $k$. 
Observe that 
$\alpha^{-1}\gamma\alpha=b^{-1}(b^{-1}a^{-1}b^{-1})b=b^{-2}a^{-1}=
a^{-r^2}b^{-2}$.
So, for some integer $k$, we must have that 
$\gamma^k=(b^{-1}a^{-1}b^{-1})^k=a^{-r^2}b^{-2}$.
From Equation \ref{eqn:bab}, this implies that:
\[a^{(r^4-r^2-r^{2k+1}+r)/(r^2-1)} = b^{2k-2}.\]
Thus $k\equiv 1 \pmod{q}$, so $r^k\equiv r \pmod{p}$. 
So, 
\[a^{(r^4-r^2-r^{2k+1}+r)/(r^2-1)} = a^{r^2-r}.\]
Thus $r^2\equiv r \pmod{p}$, which, in turn, implies that
$r\equiv 1 \pmod{p}$,
a contradiction.
\end{proof}

\begin{lemma}
The latin bitrade constructed in Theorem~$\ref{pq}$ is thin if and only
if the solutions to 
\[r^j + r^{j-1} \equiv r^{i+j-1}+1 \pmod{p}\]
are precisely 
$i\equiv j\equiv 0 \pmod{q}$
and
$i\equiv j\equiv 1 \pmod{q}$.
\label{isitthin}
\end{lemma}

\begin{proof}
By Lemma~\ref{lem:thingroupbitrade} 
the latin bitrade is 
{\em not} 
thin if and only if 
$\alpha^i \beta^j \gamma^k = 1$ has a non-trivial 
solution in $i$, $j$, $k$.
In general, we can simplify 
$\alpha^i \beta^j \gamma^k = 1$ as follows:
\begin{alignat*}{3}
&& \alpha^i \beta^j \gamma^k &= 1 \\
&\Leftrightarrow \quad \quad& b^i(ab)^j(b^{-1}a^{-1}b^{-1})^k &= 1 \\
&\Leftrightarrow &  b^{i+j} a^{r(r^j-1)/(r-1)} \left( b^{2k}a^{r(r^{2k}-1)/(r^2-1)} \right)^{-1} &= 1 \\ 
&\Leftrightarrow & b^{i+j} a^{r(r^j-1)/(r-1)} a^{-r(r^{2k}-1)/(r^2-1)} b^{-2k} &= 1 \\ 
&\Leftrightarrow & a^{r(r^j-1)/(r-1)-r(r^{2k}-1)/(r^2-1)} &= b^{2k-(i+j)}.
\end{alignat*}
Since $a$ and $b$ are elements of different prime order, 
$\alpha^i \beta^j \gamma^k = 1$ 
if and only if 
$i+j \equiv 2k \pmod{q}$
and 
\[
\]
\begin{alignat*}{3}
&&\frac{r(r^j-1)}{(r-1)} - \frac{r(r^{2k}-1)}{(r^2-1)} \equiv 0 \pmod{p} \\
&\Leftrightarrow \qquad& r((r^j-1)(r+1) - r^{2k}+1) &\equiv 0 \pmod{p}  \\
&\Leftrightarrow &  r^{j} + r^{j-1} - 1 - r^{2k-1}&\equiv 0 \pmod{p}.
\end{alignat*}
The result follows. 
\end{proof}

An example of a non-thin latin bitrade is 
the case $q=11$, $p=23$ and $r=4$, as $4^5+4^6\equiv 4^9+1 \pmod{23}$.
It is an open problem to predict when the latin bitrade in this subsection is
thin (and indeed, minimal).   
In general, if the ratio $p/q$ is large there seems to be more chance of the latin bitrade being thin. In particular, 
it can be shown that 
if $q=3$ or if $p=r^q-1$, then the latin bitrade is always thin. 

\subsection{The alternating group on $3m+1$ letters}

Let $m\geq 1$ and 
define permutations $a$ and $b$ on the 
set 
$[3m+1] = \{ 1, 2, \ldots, 3m+1 \}$:
\begin{align}
a &= (1, 2, \ldots, 2m+1) \label{eqn11} \\
b &= (m+1, m, \ldots, 1, 2m+2, 2m+3, \ldots, 3m+1) \label{eqn12}
\end{align}
So $\left| \mov(a) \cap \mov(b) \right| = m+1$.
These permutations in fact generate
the alternating group:

\begin{lemma} \label{thm:altgen}
Let $G = \langle a, b \rangle$. Then $G = A_{3m+1}$.
\end{lemma}

To prove the above lemma, we will require some results from the study of
permutation groups.  Relevant definitions can be found in
\cite{Wielandt} and
\cite{DixonMortimer}.

\begin{theorem}[\cite{Wielandt}, p. 19]
\label{thm:wielandt:9.1}
Let $G$ be transitive on $\Omega$ and $\alpha \in \Omega$. Then $G$ is
$(k+1)$-fold transitive on $\Omega$ if and only if $G_{\alpha}$ is
$k$-fold transitive on $\Omega \setminus \alpha$.
\end{theorem}

We say that $\Gamma \subseteq \Omega$
is a {\em Jordan set} and its complement $\Delta = \Omega \setminus
\Gamma$ a {\em Jordan complement} if $\left| \Gamma \right| > 1$ and
the pointwise stabiliser $G_{(\Delta)}$ acts transitively on $\Gamma$.
The next theorem is a modern version
of a result by B.~Marggraff in 1889.

\begin{theorem}[\cite{DixonMortimer}, Theorem 7.4B, p. 224 ]
\label{thm:jordancomplement}
Let $G$ be a group acting primitively on a finite set $\Omega$ of size
$n$, and suppose that $G$ has a Jordan complement of size $m$, where $m
> n/2$. Then $G \geq A_{\Omega}$.
\end{theorem}

\begin{figure}
\begin{center}
{\includegraphics{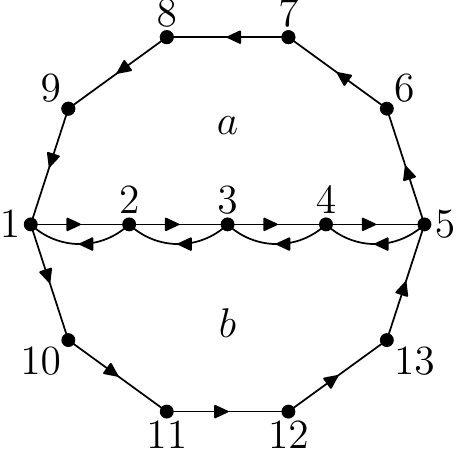}}
\end{center}
\caption{Generators $a$ and $b$ for $A_{3m+1}$ where $m=2$.}
\end{figure}

Here are a few important elements of the
group $G = \langle a, b \rangle$:
\begin{align*}
r &= [a,b] = aba^{-1}b^{-1} = (1,2m+1)(m+1,3m+1), \\
s &= (ab)^{-1}r(ab) = (1, 2m+2) (m+1, m+2), \\
t &= r s = (1, 2m+1, 2m+2) (m+1, 3m+1, m+2), \\
v_k &= a^{-k} r a^{k} = (1 a^k,(2m+1) a^k)((m+1) a^k, 3m+1 ), \\
u &= v_m = a^{-m} r a^{m} = (m, m+1) (2m+1, 3m+1).
\end{align*}

\begin{lemma}
The group $G = \langle a, b \rangle$ is primitive.
\end{lemma}

\begin{proof}
Consider the subgroup $G_1 = \{ g \in G \mid 1g = 1 \}$.
The product
\[
    ab = (m+1, m+2, \ldots, 2m+1, 2m+2, \ldots, 3m+1)
\]
is in $G_1$ so $G_1$ is transitive on 
$M=\{m+1,m+2,\dots, 3m+1\}$.
For $m=1$, $G$ is obviously $2$-transitive.
Otherwise suppose that $m > 1$. 

Since $v_m = (m+1, m)(2m+1, 3m+1)$ we see that $G_1$ is transitive on 
$M' = M \cup \{ m \}$. Using $v_{m-1}$, $v_{m-2}, \ldots, v_2 = (3,2) (m+3, 3m+1)$, 
all of which are in $G_1$, shows that $G_1$
is transitive on $[3m+1] \setminus \{ 1 \}$. By 
Theorem~\ref{thm:wielandt:9.1} with $k=1$, $G$ is $2$-transitive, and hence
primitive.
\end{proof}

\begin{proof}[Proof of Lemma~\ref{thm:altgen}]
The cases $m=1$, $2$, $3$, and $4$ can be checked by explicitly constructing an
isomorphism from $G$ to $A_{3m+1}$.  So we assume that $m \geq
5$ and define a Jordan set
\[
\Gamma = \{ 1,\, m,\, m+1,\, m+2,\, 2m+1,\, 2m+2,\, 3m+1 \} 
\]
and its complement $\Delta = [3m+1] \setminus \Gamma$. Then
$\left| \Delta \right| = 3m-6 > (3m+1)/2$ when $m \geq 5$ so $\Delta$ is large enough.  
The pointwise stabiliser
$G_{(\Delta)}$ consists of all group elements $g$ such that
$xg = x$ for $x \in \Delta$. Hence $t$ and $u$ are in $G_{(\Delta)}$ so
$G_{(\Delta)}$ is transitive on
$\Gamma$.  By
Theorem~\ref{thm:jordancomplement}, $G \geq A_{\Omega}$. 
Lastly, an odd cycle can be written as a product of an even number of
transpositions so $G \leq A_{3m+1}$.
\end{proof}

\begin{theorem} \label{3m+1} 
Let $a$ and $b$ be given as in Equations~\eqref{eqn11} and \eqref{eqn12}), and 
let $c=ab$,
$\alpha = a$, $\beta = b$, $\gamma = c^{-1}$ be elements of $G=A_{3m+1}$, the alternating group on $3m+1$ elements.
Then $\alpha$, $\beta$, and $\gamma$ 
satisfy conditions {\rm (G1)}, {\rm (G2)} and {\rm (G3)} of Theorem~$\ref{main}$. 
Thus for each $m \geq 1$, there exists a primary $(2m+1)$-homogeneous latin bitrade of size $(3m+1)!/2$
given by 
\[(
\{
(g\langle \alpha \rangle,
g\langle \beta \rangle,
g\langle \gamma \rangle) \mid g\in G \}, 
\{
(g\langle \alpha \rangle,
g\langle \beta \rangle,
g\alpha^{-1}\langle \gamma \rangle) \mid g\in G \})\]  
\end{theorem}

\begin{proof}
Clearly (G1) holds. Next we verify (G2). 
Suppose that  $\langle \alpha \rangle
\cap \langle \beta \rangle \neq 1$. Then
$a^i = b^j$ for some $i$, $j$. In particular,
$(3m+1)a^i  = (3m+1)b^j$.
Since 
$3m+1 \in \fix(a)$, it must be that
$(3m+1)b^j = 3m+1$, so $j \equiv 0 \pmod{2m+1}$. Then $a^i = 1$ so 
$i \equiv j \equiv 0 \pmod{2m+1}$, a contradiction.
By considering the action on the point $1$ we can also show that
$\langle \alpha \rangle \cap \langle \gamma \rangle \neq 1$
and
$\langle \beta \rangle \cap \langle \gamma \rangle \neq 1$.
Lastly (G3) is given by Lemma~\ref{thm:altgen}.
\end{proof}

\begin{lemma}
The latin bitrade constructed in Theorem~$\ref{3m+1}$ is thin.
\end{lemma}

\begin{proof}
Let $x$ be a point in $X = \{ 1, 2, \ldots, m \}$. Since $c^k$
fixes $x$, we have $x a^i b^j = x$ for $x \in X$.
Then it must be that $x a^i \in \{ 1, 2, \ldots, m+1 \}$ otherwise $b^j$
will be unable to map $x a^i$ onto $x$. 
Thus $i\equiv 0$ or $1 \pmod{2m+1}$. If $i\equiv 0 \pmod{2m+1}$, then
 we must have $xb^j=x$ for each $x\in X$. 
Thus $j\equiv 0 \pmod{2m+1}$ which in turn implies that $k\equiv 0 \pmod{2m+1}$.
Otherwise $i\equiv 1 \pmod{2m+1}$, which similarly implies that 
$j\equiv k\equiv 1 \pmod{2m+1}$.
\end{proof}

\begin{lemma}
The latin bitrade constructed in Theorem~$\ref{3m+1}$ is
orthogonal.
\end{lemma}

\begin{proof}
From Lemma \ref{lem:orthog},
the latin bitrade is orthogonal if and only if $|C \cap C^a| = 1$. So suppose that
$c^i = a^{-1} c^j a$
for some integers $i$, $j$. Then
\[
(1)c^i = 1 = (1) a^{-1} c^j a  =  (2m+1) c^j a.
\]
For $(2m+1) c^j a  = 1$ we must have $(2m+1) c^j = 2m+1$ so
$j \equiv 0 \pmod{2m+1}$ which also
implies that $i \equiv 0 \pmod{2m+1}$. Hence $|C \cap C^a| = 1$ as
required.
\end{proof}

\begin{example}
Letting $m=1$, we construct a thin, orthogonal latin bitrade of size $12$ as in this subsection.
Here $a=(123)$, $b=(214)$ and $c = (243)$.
We use the following cosets of $A= 
\langle a \rangle$ 
$B=\langle b \rangle$ and 
$C=\langle c \rangle$ within the alternating group $A_{4}$: 
\[
\begin{array}{ll}
\begin{array}{rl}
A &= \{ 1, (123), (132) \} \\
cA &= \{ (243), (124), (13)(24) \} \\
c^{-1}A &= \{ (234), (12)(34), (134) \}\\
bA &= \{ (142), (143), (14)(23) \} 
\end{array}
\quad \quad &
\begin{array}{rl}
B &= \{ 1, (124), (142) \} \\
aB &= \{ (123), (14)(23), (234) \} \\
a^{-1} B &= \{ (132), (134), (13)(24) \} \\
cB &= \{ (243), (12)(34), (143) \} 
\end{array} \\
& \\
\begin{array}{rl}
C &= \{ 1, (234), (243) \} \\
aC &= \{ (123), (13)(24), (143) \} \\
\end{array} & 
\begin{array}{rl}
a^{-1} C &= \{ (132), (142), (12)(34) \} \\
b^{-1} C &= \{ (124), (134), (14)(23) \} 
\end{array}  
\end{array}\]

\[
T^{\circ}\!=
\begin{tabular}{|c||c|c|c|c|}
\hline $\circ$ & $B$ & $aB$ & \!$a^{-1}B$\! & \!$cB$\! \\
\hline 
\hline $A$ & $C$ & $aC$ & \!$a^{-1} C$\! & \\
\hline $cA$ & $b^{-1}C$ & & \!$aC$\! & \!$C$\! \\
\hline \!$c^{-1} A$\! & & $C$ & $b^{-1}C$ & $a^{-1}C$  \\
\hline $bA$ & $a^{-1} C$ & \!$b^{-1} C$\! & & $aC$  \\
\hline
\end{tabular} 
\ \ T^{\star}\!= 
\begin{tabular}{|c||c|c|c|c|}
\hline $\circ$ & $B$ & $aB$ & \!$a^{-1}B$\! & \!$cB$\! \\
\hline 
\hline $A$ & $a^{-1}C$ & $C$ & \!$aC$\! & \\
\hline $cA$ & $C$ & & \!$b^{-1}C$\! & \!$aC$\! \\
\hline \!$c^{-1} A$\! & & $b^{-1}C$ & $a^{-1}C$ & $C$  \\
\hline $bA$ & $b^{-1} C$ & \!$aC$\! & & $a^{-1}C$  \\
\hline
\end{tabular} 
\]
\end{example}

\section{Minimal $k$-homogeneous latin bitrades}

Table~\ref{tableMinimalTrades} lists the sizes of the smallest minimal
$k$-homogeneous latin bitrades, where $k$ is odd and $3\leq k\leq 11$.
We give the smallest such sizes for each of Theorems \ref{p3},
\ref{pq} and \ref{3m+1}, comparing these to the smallest sizes given by
Lemma 17 and Table 2 of \cite{CaDoYa}. It is known that the smallest
possible size of a minimal $3$-homogeneous bitrade is $12$; in any case
the final column gives the smallest known example in the literature.
When applying Theorem \ref{pq} we use Lemma \ref{isitthin} to verify
that the latin bitrade is thin and thus minimal.  
For arbitrary $k$ (including even values), \cite{CaDoYa} gives
the construction of minimal, $k$-homogeneous latin bitrades of size
$\lceil1.75k^2+3\rceil k$. (This paper also improves this bound for small values of $k$.) 
If $k$ is prime then Theorem~\ref{p3} improves this
result.
\begin{table}[h]\label{tableMinimalTrades}
\[\begin{array}{|c|l|l|l|l|l|}
\hline
k & \textnormal{\rm{Thm }} \ref{p3} & \textnormal{\rm{Thm }} \ref{pq} &
\textnormal{\rm{Thm }} \ref{3m+1} & \cite{CaDoYa} & \textnormal{\rm{Smallest known}}\\
\hline
\hline
3 & 27 & 21\ (p=7,q=3,r=2) & 12 & 21 & 12  \\
\hline
5 & 125 & 55\ (p=11,q=5,r=3) & 2520 & 75 & 55 \\
\hline
7 & 343 & 203\ (p=29, q=7,r=7) & 1814400 & 133 & 133 \\
\hline
9 & \textnormal{\rm{N/A}} & \textnormal{\rm{N/A}} & 3113510400 & 243 & 243 \\
\hline
11 & 1331 & 737\ (p=67, q=11, r=14) & 16!/2 & 407 & 407 \\
\hline 
\end{array}\]
\caption{{Sizes of minimal $k$-homogeneous latin bitrades.}} 
\end{table}

\end{document}